\documentclass[12pt]{amsart}

\usepackage[all,color]{xy}
\usepackage{pb-diagram}
\usepackage[mathscr]{eucal}
\usepackage{hyperref}
\hypersetup{
    colorlinks=true, 
    linktoc=all,     
    linkcolor=blue,  
}

\usepackage{xcolor}
\usepackage[normalem]{ulem}

\DeclareMathAlphabet{\mathpzc}{OT1}{pzc}{m}{it}

\usepackage{amsfonts}
\usepackage{amsmath}
\usepackage{amssymb}

\newcommand{\dbar}{\overline{\partial}}

\newcommand{\ddbar}{\sqrt{-1}\partial\dbar}

\def\cC{{\mathcal C}}

\def\cE{{\mathcal E}}
\def\cF{{\mathcal F}}

\def\cK{{\mathcal K}}

\def\cJ{{\mathcal J}}

\def\cO{{\mathcal O}}

\def\cS{{\mathcal S}}
\def\cT{{\mathcal T}}

\newtheorem{theorem}{Theorem}[section]
\newtheorem{proposition}{Proposition}[section]
\newtheorem{lemma}{Lemma}[section]
\newtheorem{definition}{Definition}[section]
\newtheorem{corollary}{Corollary}[section]

\numberwithin{equation}{section}

\begin{document}

\title{Sup-slopes and sub-solutions for fully nonlinear elliptic equations}

\author[{Bin Guo and Jian Song}
]{Bin Guo$^*$ and Jian Song$^\dagger$  }

 \thanks{Work supported in part by the National Science Foundation under grants  DMS-2203607 and DMS-2303508, and the collaboration grant 946730 from Simons Foundation.}

\address{$^*$ Department of Mathematics \& Computer Science, Rutgers University, Newark, NJ 07102}

\email{bguo@rutgers.edu}

\address{$^\dagger$ Department of Mathematics, Rutgers University, Piscataway, NJ 08854}

\email{jiansong@math.rutgers.edu}

\begin{abstract} 

{\footnotesize } We establish a necessary and sufficient condition for solving a general class of fully nonlinear elliptic equations on closed Riemannian or hermitian manifolds, including both hessian and hessian quotient equations. It settles an open problem of Li and Urbas. Such a condition is based on an analytic slope invariant analogous to the slope stability and the Nakai-Moishezon criterion in complex geometry. As an application, we solve the non-constant $J$-equation on both hermitian manifolds and  singular K\"ahler spaces.

\end{abstract}

\maketitle


\section{Introduction}

In this paper, we aim to establish an analytic criterion for the solvability problem of fully nonlinear second-order elliptic partial differential equations on a closed Riemannian or hermitian manifold. 

The Dirichlet problem 
\begin{equation} \label{dp}
\left\{
\begin{array}{l}
F[u] =f(\lambda(\nabla^2 u))= e^\psi  , ~ {\rm in} ~\Omega,  \\
\\
u =\phi, ~{\rm on} ~\partial \Omega,
\end{array} \right.
\end{equation}%
on a domain $\Omega$ in $\mathbb{R}^n$ is studied in the classical works of \cite{CNS, Tr, Gu1, Gu1, GS} and many others, where $\psi$ is a function on $\Omega$,  $f$ is a given symmetric function on $\mathbb{R}^n$ and $\lambda=(\lambda_1, ..., \lambda_n)$ denotes the eigenvalues of $\nabla^2 u$.

We assume that $f\in C^\infty(\Gamma)\cap C^0(\overline \Gamma)$ is a symmetric function defined on an open convex symmetric cone $ \Gamma\subset \mathbb{R}^n$ with vertex at the origin and the positive cone $\Gamma_n \subset \Gamma$. Furthermore, $f$  satisfies the following standard conditions:
 \begin{eqnarray} 
 && \frac{\partial f}{\partial \lambda_i} >0 ~\textnormal{in} ~\Gamma,~ i=1, ..., n,  \label{fcon1}\\
 && f~\textnormal{is ~concave~in}~\Gamma, \label{fcon2}\\
 && f>0 ~\textnormal{in}~\Gamma, ~ f= 0 ~\textnormal{in~} \partial \Gamma,  \label{fcon3}\\
&& \lim_{R\rightarrow \infty} f(R\lambda) =\infty, ~{\rm for ~any~} \lambda \in \Gamma. \label{fcon4}
\end{eqnarray}

There are many examples  of symmetric functions satisfying (\ref{fcon1}), (\ref{fcon2}), (\ref{fcon3}), (\ref{fcon4}) such as
$$f(\lambda) = \left( S_k(\lambda) \right)^{\frac{1}{k}} = \left( \sum_{ i_1<i_2<...<i_k} \lambda_{i_1}\lambda_{i_2}... \lambda_{i_k} \right)^{\frac{1}{k}},$$ 
for  $k =1$, ..., $n$. The corresponding cone is given by 
$$\Gamma_k = \{ \lambda \in \mathbb{R}^n~|~ S_i(\lambda)>0, ~i =1, ..., k\}. $$
In particular, equation (\ref{dp}) corresponds to the Monge-Amp\`ere equation when $k=n$ and the Laplace equation when $k=1$. 
Another  interesting example is 
$$f(\lambda) =\left( S_{k, l}(\lambda)\right)^{\frac{1}{k-l}} = \left( \frac{S_k(\lambda)}{S_l(\lambda)} \right)^{\frac{1}{k-l}}, ~ 1\leq l < k \leq n.$$
In this case, equation (\ref{dp}) becomes a hessian quotient equation.

It is proved by Trudinger \cite{Tr} that equation (\ref{dp})  can be uniquely solved if the principle curvatures $\nu_1$, ..., $\nu_{n-1}$ of $\partial \Omega$ satisfy $(\nu_1, ..., \nu_{n-1}, R) \in \Gamma$ for some $R>0$. A sub-solution criterion is also established based on the works of Guan and Guan-Spruck \cite{Gu1, GS}.

We will investigate   equation (\ref{dp}) in the global setting, which is initiated by the work of Li \cite{Li} on closed Riemannian manifolds.  Let $(M, g)$ be an $n$-dimensional closed Riemannian manifold. Suppose $\theta$ is a smooth global symmetric $2$-tensor. For any $u\in C^\infty(M)$, 
 $$\theta_u= \theta + \nabla^2 u$$
 is also a global symmetric $2$-tensor, where $\nabla^2$ is the hessian operator associated to the Riemannian metric $g$. 
   Let $\lambda (\theta_u)= (\lambda_1, ..., \lambda_n)$ be the eigenvalues of $\theta_u$ with respect to $g$ on $M$. At each point $p$, one can choose normal coordinates $(x_1, x_2, ..., x_n)$ for $g$ near $p$ so that 
 $$g_{ij}(p) = \delta_{ij}, ~ (\theta_u)_{ij} = \lambda_i \delta_{ij}, ~i, j =1, 2, ..., n. $$

 We will consider the following hessian type equation
  \medskip
 \begin{equation}\label{maineqn}
 F[u] =  f(\lambda(\theta_u))= e^{\psi + c} , ~c\in \mathbb{R}
 \medskip
 \end{equation}
on $M$ for any given $\psi\in C^\infty(M)$. 
In the case when $\theta = g$, the equation becomes
$$F[u] = f(I+ g^{-1}\cdot\nabla^2 u) = e^{\psi + c}$$ studied by \cite{Li} and later by by \cite{De2},  \cite{Ur} and  \cite{Gu2} under various curvature, structural and sub-solution assumptions. The solvability problem of (\ref{maineqn}) is raised by Urbas \cite{Ur}. The major progress is made in the general case by the important work of Szekelyhidi \cite{Sz}, where the notion of $\cC$-subsolution is introduced and a priori estimates are established. There have also been extensive and deep works  by \cite{Gu2, Gu3} in terms of sub-solutions. Such  a priori estimates, however, cannot be directly applied to establish the solvability of equation (\ref{maineqn}) in general,  because most of such structural or sub-solution conditions do not seem to be preserved along the deformation of the continuity method. This particularly poses challenges when there are geometric or analytic obstructions to the equation.

Our approach is to introduce a set of new sub-solutions paired with an invariant for equation (\ref{maineqn}) as an analogue of the slope invariants in complex geometry. They are inspired by sub-solutions constructed in \cite{SW1} for Donaldson's $J$-equation and the Nakai-Moishezon criterion in algebraic geometry.

 We begin by letting $\cE$  be the set of admissible functions for equation (\ref{maineqn}) associated to the cone $\Gamma$  defined by
 \begin{equation}\label{eqn:e definition}
 \cE  =\cE  (M, g, \theta, f, \Gamma)= \{ u \in C^\infty(M)~|~ \lambda(\theta_u(p)) \in \Gamma, ~{\rm for~any}~ p\in M \}.
 \end{equation}
We assume that $\mathcal E$ is not empty. The following definition is a key ingredient in our approach. 
 \begin{definition}\label{gslope} 

 The sup-slope $\sigma$ for equation (\ref{maineqn})  associated to the cone $\Gamma$ is defined to be 
 \begin{equation}\label{supslo}
 \sigma = \inf_{ u \in \cE} \max_M   e^{-\psi} F[u]   \in [0, \infty). 
 \end{equation} 
 \end{definition}
\noindent In fact, $\sigma$ is always positive whenever $\cE\neq \emptyset$. We define $$f_{\infty, i}: \Gamma  \rightarrow (0, \infty]$$ by  
\begin{equation}
f_{\infty, i} (\lambda_1, ..., \lambda_i, ...,      \lambda_{n}) = \lim_{R\rightarrow \infty} f(\lambda_1, ..., \lambda_{i-1},  R, \lambda_{i+1}, ..., \lambda_n)
\end{equation}
for $i = 1, 2, ..., n$ and 
$$f_\infty (\lambda) = \min_{i=1, ..., n}  f_{\infty, i}(\lambda) .$$
In fact, either $f_\infty \equiv \infty$ on $\Gamma$ or $f_\infty$ is a positive symmetric concave function in $\Gamma$ satisfying (\ref{fcon1}) and  (\ref{fcon4}). 
Equivalently, one can define $f_\infty$ by
$$
%
f_\infty(\lambda)=f_\infty(\lambda_1,..., \lambda_n) = \lim_{R \to \infty} f(\lambda_1, ..., \lambda_{j-1}, R, \lambda_{j+1}, ..., \lambda_n), 
%
$$
 if $\lambda_j=\max \{ \lambda_1, \lambda_2, ..., \lambda_n\}$.

%
%
\begin{definition}The global subsolution operator 
$$F_\infty: \cE \rightarrow C^0(M) \cup \{\infty\}$$ 
is define by
\begin{equation}
F_\infty[u] = f_\infty(\lambda(\theta_u)). 
\end{equation}
\end{definition}
Although $f_{\infty, i} \left( \lambda (\theta_{\underline u} ) \right)$ depends on the local coordinates and the choice of $1\leq i\leq n$,   $  f_{\infty} \left( \lambda (\theta_{\underline u} )  \right)$ is indeed a globally defined function on $M$.%
 The following  are a set of new sub-solutions paired with the sup-slope for equation (\ref{maineqn}). 

\begin{definition} \label{subdef} Let $\sigma$ be the sup-slope for equation (\ref{maineqn}) in Definition \ref{gslope}. Then $\underline u \in \cE $ is said to be a  sub-solution  associated with  $\sigma$, if 
\begin{equation}\label{subsol}
  e^{-\psi} F_\infty[ \underline u]   > \sigma, 
\end{equation}
on $M$.
\end{definition}
The above sub-solution follows from the idea in \cite{SW1} by replacing the maximum eigenvalue of $\theta_{\underline u}$ by $\infty$.  We now can state our main result on the necessary and sufficient conditions for solving equation (\ref{maineqn}) in terms of sup-slopes and sub-solutions.

\begin{theorem}\label{mainthm1} Let $(M, g)$ be a closed Riemannian manifold of $\dim M=n$. Suppose $\theta$ is a smooth symmetric $2$-tensor, $\psi \in C^\infty(M)$ and $f$ satisfies (\ref{fcon1}),  (\ref{fcon2}),  (\ref{fcon3}),  (\ref{fcon4}) for an open convex symmetric cone $ \Gamma \subset \mathbb{R}^n$ containing $\Gamma_n$.  Then the following are equivalent. 

\begin{enumerate} 

\item There exists a smooth solution $u \in \cE $ solving equation (\ref{maineqn}), or equivalently, 
$$ e^{-\psi} F[u] = constant. $$

\item  There exists a sub-solution $\underline u \in \cE$ for equation (\ref{maineqn}) associated with the sup-slope $\sigma$, i.e., 
$$e^{-\psi} F_\infty[\underline u] > \sigma .$$

\item There exists  $u \in \cE$ satisfying
$$ \max_{ M} e^{-\psi}  F[u]  <  \min_{M  } e^{-\psi } F_{\infty}[u].$$

\item There exist  $\overline{u}, \underline u \in \cE$ satisfying
$$ \max_M e^{-\psi} F[ \overline u]  <  \min_{M  } e^{-\psi} F_{\infty} [\underline u] .$$

\end{enumerate}
Furthermore, if $u \in \cE$ solves equation (\ref{maineqn}), $u$ is unique up to a constant and   
$$F[u] = \sigma e^\psi,$$
where $\sigma$ is the sup-slope.

\end{theorem}

Condition (4) of Theorem \ref{mainthm1} can be viewed as a pair of super and sub solutions. 
An obvious necessary condition for solving  equation (\ref{maineqn}) is 
\begin{equation}\label{semist}
\sup_{u\in \cE} \min_M e^{-\psi} F[u] =\inf_{u\in \cE} \max_M  e^{-\psi} F[u].
\end{equation}
In general, (\ref{semist}) is not a sufficient condition because it corresponds to a semi-stable condition. However, in many cases, one still expects a unique solution with mild singularities to (\ref{maineqn}) under the assumption of  (\ref{semist}) (c.f. \cite{DMS}).  
A significant consequence of Theorem \ref{mainthm1} is that  the unique normalizing constant $c$ in equation (\ref{maineqn}) can be identified as
$$e^c =  \sigma$$
in relation to the sup-slope.

Yau's solution \cite{Y} to the Calabi  conjecture initiated the study of global complex equations of hessian type and tremendous progress has been made in the past decades, e.g. \cite{Kol, HMW, DK, Sz} and many others.  There  also has  been considerable progress recently in PDE methods for $L^\infty$-estimates of fully non-linear equations \cite{GPT, GP}  and new geometric estimates on singular complex spaces \cite{GPSS1, GPSS2}. 

We now consider the complex analogue of Theorem \ref{mainthm1}. Let $(X, g)$ be a closed complex manifold of $\dim_{\mathbb{C}} X= n$ equipped with a hermitian metric $g$. The metric $g$ is uniquely associated to a positive definite hermitian form
$$\omega= \sqrt{-1} \sum_{i, j=1}^n g_{i\bar j} dz_i \wedge d\bar z_j$$
in local holomorphic coordinates. 
Suppose $\theta $ is a smooth hermtian $(1,1)$-form on $X$ given by $\theta= \sqrt{-1} \sum_{i, j=1}^n \theta_{i\bar j} dz_i \wedge d\bar z_j $ in local holomorphic coordinates.  For any $u\in C^\infty(X)$, 
 $$\theta_u = \theta + \ddbar u=  \sqrt{-1} \sum_{i, j=1}^n \left (\theta_{i\bar j} + u_{i\bar j} \right) dz_i \wedge d\bar z_j$$
 is also a global hermitian $(1,1)$-form.    Let $\lambda (\theta_u)= (\lambda_1, ..., \lambda_n)$ be the eigenvalues of $\theta_u$ with respect to $\omega$ on $X$. 
 %
 
 Similar to the discussion in the Riemannian setting, we can consider  the following complex hessian type equation
 \begin{equation}\label{maineqn2}
 F[u] = f(\lambda(\theta_u))=  e^{\psi + c}, ~ c\in \mathbb{R}
 \end{equation}
for any given $\psi\in C^\infty(X)$.

We define the set of admissible functions by
 $$\cE =\cE (X, g, \theta, f, \Gamma)= \{ u \in C^\infty(X)~|~ ~\lambda(\theta_u(p)) \in \Gamma, ~{\rm for~all}~p\in X \}. $$ 
 and the sup-slope $\sigma$ for equation (\ref{maineqn2}) by
 $$\sigma = \inf_{ u \in \cE} \max_X  e^{-\psi} F[u]   \in [0, \infty) $$
 analogous to Definition \ref{gslope}. 
Similarly, $\underline u \in \cE$ is said to be a  sub-solution associated to the sup-slope $\sigma$ for equation (\ref{maineqn2}), if 
$$    e^{-\psi }  F_\infty [\underline u]  > \sigma$$
 on $X$.

We have the following result almost identical to Theorem \ref{mainthm1} in the hermitian case. 

\begin{theorem}\label{mainthm2} Let $(X, \omega)$ be a  closed hermitian manifold  of $\dim_{\mathbb{C}} X =n$. Suppose $\theta$ is a smooth hermitian form, $\psi \in C^\infty(X)$, and $f$ satisfies (\ref{fcon1}),  (\ref{fcon2}),  (\ref{fcon3}),  (\ref{fcon4}) for an open convex symmetric cone $\Gamma \subset \mathbb{R}^n$ containing $\Gamma_n$.  Then the following are equivalent.  
\begin{enumerate} 

\smallskip
\item There exists a smooth solution $u \in \cE $ solving equation (\ref{maineqn2}).

\medskip

\item  There exists a sub-solution  for equation (\ref{maineqn2}) associated with the sup-slope $\sigma$.

\medskip

\item There exists  $u \in \cE$ satisfying
$$ \max_X e^{-\psi}  F[u]  <  \min_{ X } e^{-\psi } F_{\infty} [u].$$

\item There exist  $\overline{u}, \underline u \in \cE$ satisfying
$$ \max_X e^{-\psi} F[\overline u]  <  \min_{ X  } e^{-\psi} F_{\infty} [\underline u] .$$

\end{enumerate}
Furthermore, if $u \in \cE$ solves equation (\ref{maineqn2}), $u$ is unique up to a constant and   
$$F[u] = \sigma e^\psi,$$
where $\sigma$ is the sup-slope for equation (\ref{maineqn2}).

\end{theorem}

  Theorem \ref{mainthm2} can be applied to determine the unique normalizing constant for complex Monge-Amp\`ere equations on a hermitian manifold. 
 
\begin{corollary} \label{detcon}Let $(X, \omega)$ be a closed hermitian manifold of $\dim_{\mathbb{C}} X =n$.   For any $\psi\in C^\infty(X)$, let $\varphi \in C^\infty(X)\cap {\rm PSH}(X, \omega)$ be the unique solution to the complex Monge-Amp\`ere equation 
\begin{equation}\label{maeqn}
(\omega + \ddbar \varphi)^n = e^{\psi + c} \omega^n, ~ \max_X \varphi =0 
\end{equation}
for some $c\in \mathbb{R}$. Then 
$$ e^c = \inf_{u \in C^\infty(X)\cap {\rm PSH}(X, \omega)} \max_X  \frac{  e^{-\psi} (\omega + \ddbar u)^n } { \omega^n}. $$

\end{corollary}

 If $\omega$ is a K\"ahler form, the normalizing constant $c$ is explicitly given by
$$e^c = \frac{\int_X \omega^n}{\int_Xe^\psi \omega^n} .$$
Equation (\ref{maeqn}) is uniquely solved in  \cite{TW}, however it has been an open question how to calculate the constant $c$ in equation (\ref{maeqn}), when the hermitian form $\omega$ is not closed.  %
Corollary \ref{detcon} could have applications in solving degenerate complex Monge-Amp\`ere equations on hermitian manifolds in relation to a conjecture of Demailly-Paun \cite{DP} on the Fujiki class.

 We will also apply Theorem \ref{mainthm2} to complex hessian quotient equations. The $J$-equation introduced by Donaldson \cite{Do} is a special hessian quotient equation on a closed K\"ahler manifold with constant on the right hand side. More precisely, let $X$ be a closed K\"ahler manifold of $\dim_{\mathbb{C}}X=n$ with two K\"ahler classes $\alpha$ and $\beta$. For a given K\"ahler form $\omega\in \beta$, one looks for a K\"ahler form $\theta \in \alpha$ satisfying
\begin{equation}\label{constj}
\frac{\theta^{n-1}\wedge \omega}{\theta^n} = c,  
\end{equation}
where $c= \frac{\alpha^{n-1}\cdot \beta}{\alpha^n}$ is a complex invariant due to the K\"ahler conditions of $\alpha$ and $\beta$. 
It is proved in \cite{W1, SW1} that  equation (\ref{constj}) can be solved if and only if there exists a smooth sub-solution $\underline \theta \in \alpha$, which is a K\"ahler form satisfying
\begin{equation}\label{swsub}
(n-1) \underline \theta^{n-2} \wedge \omega < nc \underline\theta^n.
\end{equation}
It is conjectured \cite{LS}  that the sub-solution condition (\ref{swsub}) is equivalent to  the following nonlinear Nakai-Moishezon criterion: for any $m$-dimensional proper subvariety $Z$ of $X$, 
\begin{equation}\label{LScon}
\left. m \frac{\alpha^{m-1}\cdot \beta}{\alpha^m} \right|_Z  < n \frac{\alpha^{n-1}\cdot \beta}{\alpha^n}.
\end{equation} 
The above condition can also be viewed as an analogue of the slope stability for hermitian bundles over K\"ahler manifolds. The conjecture is fully settled by \cite{ChG, DaP, So} in the K\"ahler case. However, the solvability of the non-constant $J$-equation for a given $\psi\in C^\infty(X)$
\begin{equation}\label{nconstj}
\frac{\theta^{n-1}\wedge \omega}{\theta^n} = e^{\psi+c},  ~c\in \mathbb{R}
\end{equation}
has been open for the past ten years.  The $J$-equation can be further defined on a hermitian manifold  $X$  with constant or non-constant functions on the right hand side. This has also been extensively studied  and a sufficient condition is obtained  by Sun \cite{Su} as a generalization of (\ref{swsub}).   We can now settle these problems by applying Theorem \ref{mainthm2}.  The sup-slope and sub-solution criterion directly give a necessary and sufficient condition for equation (\ref{nconstj}), since it is equivalent to the standard hessian quotient equation $\frac{n S_n}{S_{n-1}}=e^{-\psi - c}$. 

More precisely,  we let $X$ be a closed hermitian manifold of $\dim_{\mathbb{C}}X=n$ equipped with two hermtian metrics $\omega$ and $\chi$.  Given $\psi\in C^\infty(X)$, we define the $J$-slope $\xi$ for $(X, \omega, \chi, \psi)$  by 
\begin{equation}\label{jslop0}
\xi = \xi(X, \omega, \chi, \psi) = \sup_{u\in C^\infty(X)\cap {\rm PSH}(X, \chi) } \min_ X \left(  e^{-\psi} \frac{\chi_u^{n-1}\wedge \omega}{\chi_u^n} \right),
\end{equation}
where $\chi_u = \chi+ \ddbar u$. Then the following theorem immediately follows from Theorem \ref{mainthm2}.

\begin{theorem} \label{mainthm3} Suppose $X$ is a closed hermitian manifold of $\dim_{\mathbb{C}}X=n$. For any $\psi\in C^\infty(X)$ and  smooth hermtian metrics $\omega$ and $\chi$, the hermitian $J$-equation
\begin{equation}\label{hjeqn}
\frac{\chi_u^{n-1}\wedge \omega}{\chi_u^n} = e^{\psi+ c} 
\end{equation}
 admits a solution $u \in C^\infty(X)\cap {\rm PSH}(X, \chi)$ if and only if there exists a sub-solution $\underline u \in C^\infty(X)\cap {\rm PSH}(X, \chi)$ associated with the $J$-slope (\ref{jslop0}) satisfying
\begin{equation}
(n-1) e^{-\psi} \chi_{\underline u}^{n-2}\wedge \omega < n \xi  \chi_{\underline u} ^{n-1}.
\end{equation}
In particular, the solution is unique up to a constant and the constant $c$ in (\ref{hjeqn}) is given by $e^c = \xi$. 

\end{theorem}

We return to the K\"ahler case where singularities are allowed for the underlying variety. The Lejmi-Szekelyhidi conjecture is proved in \cite{So} for the $J$-equation (\ref{constj}) on a normal K\"ahler variety embedded in an ambient (open) K\"ahler manifold.  Naturally, one would wish to solve the  $J$-equation  (\ref{nconstj}) on a normal K\"ahler variety.  

 Let $X$ be an $n$-dimensional normal projective variety with two K\"ahler classes $\alpha$ and $\beta$. For any smooth K\"ahler metric $\omega \in \beta$ and $\psi\in C^\infty(X)$ (the smoothness is defined via local or global holomorphic embeddings), we can extend the definition of  the $J$-slope to a normal variety $X$ by
 \begin{equation}\label{jslop}
 \xi= \xi(X, \alpha, \omega, \psi)= \sup_{\chi\in   \cK(\alpha) } \inf_ {X^\circ} \left( e^{-\psi} \frac{\chi^{n-1}\wedge \omega}{\chi^n} \right), 
 \end{equation}
where $\cK(\alpha)$ is the set of all smooth K\"ahler metrics in $\alpha$, and $X^\circ$ is the regular part of $X$. We establish a sufficient condition for solving equation (\ref{nconstj}) on $X$.

\begin{theorem} \label{mainthm4}  Let $X$ be a  normal projective variety of $\dim_{\mathbb{C}} X=n$.   For $\psi \in C^\infty(X)$, a K\"ahler class $\alpha$ and a smooth K\"ahler metric $\omega$ on $X$,   there exists a   K\"ahler current $\theta\in \alpha$ with bounded local potentials solving the $J$-equation
\begin{equation}\label{sinjeqn}
\theta^{n-1} \wedge \omega  = e^{\psi+c} \theta^n
\end{equation}
on $X$ for some $c \in \mathbb{R^+}$, if  there exist $\epsilon>0$ and a  smooth K\"ahler form  $\underline\theta \in \alpha$ satisfying 
\begin{equation}\label{sinswsub} 
(n-1)  \underline\theta^{n-2}\wedge \omega \leq n (\xi-\epsilon) e^{\psi} \underline\theta^{n-1}, 
\end{equation}
where $\xi$ is the $J$-slope in \eqref{jslop}.

\end{theorem}

Both (\ref{sinjeqn}) and (\ref{sinswsub}) are defined in the sense of currents. We expect that the solution is unique  and  smooth on $X^\circ$ with $\xi = e^c$.  It will require a uniform second order estimate away from the singular set of $X$. We also believe that the $J$-slope (\ref{jslop}) should be closely related to the minimal slope introduced in \cite{DMS}  for the constant $J$-equation. Finally, we remark that the notion of the sup-slope and sub-solution is fairly flexible and it can be extended to a larger family of globally fully nonlinear equations by modifying the conditions for $f$ (c.f. \cite{Sz}). For example,  the work of  \cite{CJY} could be generalized for the  deformed hermitian Yang-Mills equation with a prescribed angle function, where $f$ is given by
$$f(\lambda) = \sum_{i=1}^n {\rm arctan}(\lambda_i). $$ 
In particular, one can similarly define the sup-slope and a pairing sub-solution in the supercritical case.

%
%
%


\section{The continuity method}

We will begin our proof for Theorem \ref{mainthm1}. The proof can be used for Theorem \ref{mainthm2} in the hermitian case without changes. Our goal is to show 
$$(4) \Rightarrow (1) \Rightarrow (3) \Rightarrow (2) \Rightarrow (4)$$
in Theorem \ref{mainthm1}.
The main step is to establish $(4) \Rightarrow (1)$.  In this section, we will set up the continuity method for the proof of Theorem \ref{mainthm1}.

Let $\overline{u}$ and $\underline{u} \in \cE$ be a pair of super and sub-solutions satisfying
\begin{equation}\label{supsub}
e^{\overline c}=\max_M e^{-\psi} f(\lambda(\theta_{\overline{u}}))  <  \min_{M} e^{-\psi} f_{\infty} (\lambda(\theta_{\overline{u}})) .
\end{equation}
Let $\sigma$ be the sup-slope for equation (\ref{maineqn}). Then  $$\sigma \leq e^{\overline c}$$
by definition. There exists $\delta>0$ such that 
\begin{equation}\label{gap}
e^{\overline c} < (1+\delta) e^{ \overline c} \leq  \min_{M} e^{-\psi} f_{\infty} (\lambda(\theta_{\overline{u}})). 
\end{equation}
By letting $\overline{\psi} = \log F[\overline{u}] $, we have
 $$F[\overline u] = e^{\overline{\psi}} \leq e^{\psi + \overline c},$$
by the choice $\overline c$.   Immediately we have the following lemma.

\begin{lemma}
\begin{equation}\label{barpsi}
  \overline \psi  \leq  \psi + \overline c. 
  \end{equation}
  
\end{lemma}

We will consider the following continuity method 
\begin{equation}\label{contim}
F[\overline{u} + \phi_t] = e^{\psi_t + c_t}, ~ \psi_t = (1-t)\overline{\psi} + t \psi, 
\end{equation}
for $t\in [0, 1]$, where 
$$\overline u+ \phi_t\in \cE, ~ \mathrm{sup}_M \phi_t =0$$
  and $c_t$ is a constant for each $t\in [0, 1]$. 

Let 
\begin{equation}\label{solvT}
\cT=\{ t\in [0, 1]~|~ (\ref{contim}) ~\textnormal{admits~a~smooth~solution~ at} ~t\}. 
\end{equation}

\begin{lemma} $\cT$ is  non-empty and open.
\end{lemma}

\begin{proof} The equation (\ref{contim}) can be solved for $t=0$ with $\phi_0 = 0$ and $c_0=0$.  Therefore $\cT$ is non-empty. The openness follows from the inverse function theorem by the work of \cite{De1}. In the hermitian case, the openness follows by a conformal change that turns $g$ to a Gauduchon metric. 
\end{proof}

\begin{lemma} \label{ctbd} There exists $C>0$ such that for all $t\in \cT$, 
$$ - C \leq c_t \leq t \overline c. $$

\end{lemma}

\begin{proof}  
Suppose $t\in \cT$ and $\phi_t$ is a solution of (\ref{contim}). By the maximum principle and condition (\ref{fcon1}), we have at the maximal point $p_t\in M$ of $\phi_t$, 
$$  e^{(1-t) \overline \psi + t \psi + c_t} = F[\overline u +\phi_t ] \leq F[\overline u]  = e^{\overline \psi}. $$
Therefore by (\ref{barpsi}),
%
%
$$c_t \leq t (\overline \psi - \psi) (p_t) \leq t \overline c. $$
At the minimal point of $\phi_t$, we have
$$ e^{(1-t) \overline \psi + t \psi + c_t} = F[\overline u +\phi_t ] \geq F[\overline u]  = e^{\overline \psi}.$$
Then 
$$c_t \geq t \min_M (\overline \psi - \psi) \geq - |\overline \psi|_{C^0(M)} - |\psi|_{C^0(M)}. $$
We have completed the proof of the lemma.
\end{proof}


\section{Sup-slopes and sub-solutions}

We will keep the same notations in Section 3.  We first prove some basic properties of $f_\infty$ and the subsolution operator $F_\infty$.

\begin{lemma}\label{dico} Either $f_\infty\equiv \infty$ in $\Gamma$ or  $f_\infty(\lambda)$ is bounded for each $\lambda\in \Gamma$. 
\end{lemma}

\begin{proof} It suffices to prove the lemma for $f_{\infty, n}$ since $f$ is symmetric. Suppose $f_{\infty, n}(\lambda) < \infty$ for some $\lambda \in \Gamma$. For any $\lambda' \in \Gamma$, we pick sufficiently small $\varepsilon>0$ so that $(1+\varepsilon) \lambda -\varepsilon \lambda'\in \Gamma$. We  then  define  
$$h_R(t) = f( (1-t)(\lambda_1, ...,  \lambda_{n-1}, R) + t (\lambda'_1, ..., \lambda'_{n-1}, R) ).$$
and
$$ h_\infty(t)= \lim_{R\to \infty} h_R(t) $$
for $t\in [-\varepsilon, 1]$ and $R\in [\max(\lambda_n, \lambda'_n), \infty) $. 
Obviosly, $h_R(t)$ is concave on $[-\varepsilon, 1]$, and by concavity, 
$$h_R(1) \leq \varepsilon^{-1} \left( h_R(0)- h_R(-\varepsilon) \right) + h_R(0) \leq  (1+ \varepsilon^{-1} ) h_R(0).$$
Therefore $$f_{\infty, n}(\lambda')=h_\infty(1) \leq  (1+ \varepsilon^{-1} ) h_\infty(0), $$
and the lemma is proved.%
\end{proof}

\begin{lemma}\label{dico2} If $f_\infty \neq \infty$, then it is a symmetric continuous function on $\Gamma$ satisfying the following. 
\begin{eqnarray} 
 && f_\infty=f_\infty(\lambda_1, ..., \lambda_n) ~{\rm is ~ increasing ~ in~} \lambda_i,~ i=1, ..., n,  \label{ffcon1}\\
 && f_\infty ~\textnormal{is ~concave~in}~\Gamma, \label{ffcon2}\\
 && f>0 ~\textnormal{in}~\Gamma,  \label{ffcon3}\\
&& \lim_{R\rightarrow \infty} f_\infty(R\lambda) =\infty, ~{\rm for ~any~} \lambda \in \Gamma. \label{ffcon4}
\end{eqnarray}
 
\end{lemma}

\begin{proof}  Condition (\ref{ffcon1}),  (\ref{ffcon2}), (\ref{ffcon3}) and (\ref{ffcon4}) follow directly from the properties of $f$. For example, to show (\ref{ffcon2}), we first note that $f_{\infty, n}$ is concave because
\begin{eqnarray*}
f_{\infty,n}(\lambda) + f_{\infty, n}(\lambda') &=& \lim_{R\to \infty}\left( f(\lambda_1, ..., \lambda_{n-1}, R)  +  f(\lambda'_1, ..., \lambda'_{n-1}, R) \right) \\
&\leq&2 \lim_{R\to \infty} f\left(\frac{\lambda_1+\lambda_1'}{2} , ..., \frac{\lambda_n+\lambda_n'}{2}, R \right)  \\
&=& f_{\infty, n} \left(\frac{\lambda+\lambda'}{2} \right).
\end{eqnarray*}
By symmetry, $f_{\infty, i}$ is concave for each $i=1, ..., n$. Hence $f_\infty = \min_{i=1, ..., n} f_{\infty, i}$ is again concave. 
\end{proof}

\begin{lemma}  Either $F_\infty[u] \equiv \infty$ on $M$ for each $u\in \cE$ or $F_\infty[u]\in C^0(M)$ for each $u\in \cE$.

\end{lemma}

\begin{proof}  It suffices to prove the case when $f_\infty \neq \infty$ by Lemma \ref{dico}.  $F_\infty[u] = f_\infty(\lambda(\theta_u))$ is composition of continuous functions because  $u\in C^\infty(M)$ and $f_\infty \in C^0(\Gamma)$. The lemma is then proved. 
\end{proof}

The following properties of the sup-slope are obvious from its definition. 

\begin{lemma} For any $\psi \in C^\infty(M)$ and $C\in \mathbb{R}$, the sup-slope satisfies 
$$\sigma(M, g, \theta, f, \psi+C) = e^{-C} \sigma(M, g, \theta, f, \psi)>0.$$
Moreover, for any $\psi_1, \psi_2\in C^\infty(M)$, we have 

$$e^{   - \max_M |\psi_2 -   \psi_2|   } \leq \frac{\sigma(M, g, \theta, f,  \psi_2)}{\sigma(M, g, \theta, f, \psi_1)} \leq e^{ \max_X |\psi_2 - \psi_1| }. $$
\end{lemma}
\begin{proof}
We need only to prove that $\sigma(M, g, \theta, f, \psi)>0$. By the assumption that the set  $\mathcal E$ in \eqref{eqn:e definition} is nonempty, we can pick a fixed $u_0\in \mathcal E$. It is clear that $\min_M e^{-\psi} F[u_0]>0$ since $\lambda (\theta_{u_0})\in \Gamma$ at any point of $M$. Then for any $u\in\mathcal E$, we look at the minimum point $p\in M$ of $\phi= u - u_0$. At $p$, we have 
$$e^{-\psi}F[u]\ge e^{-\psi} F[u_0]. $$ Hence $\max_M e^{-\psi}F[u] \ge \min_M e^{-\psi} F[u_0]>0$. Taking infimum over $u\in\mathcal E$ gives the desired inequality. 
\end{proof}

The following proposition is a key ingredient of our proof for Theorem \ref{mainthm1}.

\begin{proposition} \label{subpsi} Let $\delta>0$ be the fixed constant in (\ref{gap}) and $\underline u \in \cE$ be the sub-solution defined in (\ref{supsub}). Then for any $t\in \cT$, we have
\begin{equation}
 \min_{M} \left( e^{-\psi_t } f_{\infty} (\lambda(\theta_{\underline u })) \right) \geq (1+\delta) e^{c_t}, 
\end{equation}

\end{proposition}

\begin{proof} By (\ref{supsub}) and Lemma \ref{ctbd}, we have 
\begin{eqnarray*}
&& \min_{M} \left( e^{-\psi_t -c_t} f_{\infty} (\lambda(\theta_{\underline u })) \right)  \\
&=&\min_{M} \left( e^{-(1-t)\overline \psi - t\psi -c_t} f_{\infty} (\lambda(\theta_{\underline u })) \right) \\
&\geq &\min_{M} \left( e^{- \psi  -(1-t) \overline c - c_t} f_{\infty} (\lambda(\theta_{\underline u })) \right) \\
&\geq & (1+\delta) e^{- c_t -(1-t) \overline c  }   \max_M e^{-\psi} F[\overline u]  \\
&= & (1+\delta) e^{- c_t -(1-t) \overline c + \overline c }     \\
&\geq & (1+\delta)    .
\end{eqnarray*}

\end{proof}

Let us recall the definition of  $\cC$-subsolutions introduced by Szekleyhidi \cite{Sz} (see also \cite{Gu2, PT}).

\begin{definition}  Let $h\in C^\infty(M)$ be a positive function and 
$$\Gamma^{h(p)} = \{ \lambda \in \Gamma~|~ f(\lambda) \geq h(p) \}, ~ p \in M. $$

\begin{enumerate}

\item   $u\in \cE$ is said to be a $\cC_{h} $-subsolution if the set 
$$
\left( \lambda(\theta_u (p))   + \Gamma_n \right) \cap \partial \Gamma^{h( p)} \subset \mathbb{R}^n
$$ is bounded for each $p\in M$.

\medskip

\item  $u\in \cE$ is said to be a $\cC_{h, r, R}$-subsolution for some $r>0$ and $R>0$,  if 
\begin{equation}
\left( \lambda(\theta_u(p)) - r \mathbf{1} + \Gamma_n \right) \cap \partial \Gamma^{h (p)} \subset B(0, R) \subset \mathbb{R}^n
\end{equation}
for all $p\in M$, where $\mathbf{1}=(1,...,1)\in \mathbb{R}^n$. 
\end{enumerate}

\end{definition}

Obviously, a $\cC_{h, r, R}$-subsolution must also be a $\cC_h$-subsolution.  Suppose $\overline{u}$ and $\underline{u}$ are a pair of super and sub-solutions for equation (\ref{maineqn}) chosen in (\ref{supsub}).

\begin{proposition} \label{pinsubsol} There exist $r>0$, $R>0$ such that  $\underline{u}$ is a $\cC_{e^{\psi_t+c_t}, r, R}$-subsolution for (\ref{contim}) for all $t\in \cT$.

\end{proposition}

\begin{proof}    By Proposition \ref{subpsi}, there exists $\delta>0$ such that for all $t\in \cT$, 
$$\min_{M}\left( e^{-\psi_t - c_t} f_{\infty}(\lambda(\theta_{\underline{u}})) \right) \geq 1+\delta. $$ We will prove by contradiction. Suppose there exist a sequence of $q_k \in M$, $t_k\in \cT$ and $V_k \in \Gamma_n \subset \mathbb{R}^n$ such that 
$$ f(\lambda( \theta_{\underline u}(q_k)) - r\mathbf{1} + V_k) = e^{\psi_{t_k}  +c_{t_k}}, $$ 
$$\lim_{k\rightarrow \infty} V_k=\infty. $$ After passing to a subsequence, we can assume that 
$$q_k \rightarrow q \in M, ~t_k \rightarrow  t' \in [0, 1], c_{t_k} \rightarrow  c'  $$
since $c_t$ is uniformly bounded for all $t\in \cT$. Without loss of generality, we can aslo assume that the $n^{th}$ component of $V_k$ tends to $\infty$.

We will break our argument into the following two cases since $f_{\infty}(\lambda)$ 
is either $\infty$ for all $\lambda\in \Gamma $  or finite for all  $\lambda \in \Gamma$ by Lemma \ref{dico}. 
\begin{enumerate}

\item Suppose $f_{\infty}$ is unbounded.  Then $ e^{-\psi_t(p) - c_t} f_{\infty}(\lambda(\theta_{\underline{u}}(p))) =\infty$ for all $p\in M$ and $t\in \cT$.  Since $M$ is closed and $\underline u$ is smooth, there exists sufficiently small $r>0$ such that  for any $p\in M$. 
$$\lambda(\theta_{\underline{u}}(p)) - 2r \mathbf{1} \in \Gamma.$$
Then 
\begin{eqnarray*}
&&e^{\psi_{ t'} (q) +  c'} \\
&= & \lim_{k\rightarrow \infty} e^{\psi_{t_k}(q_k) + c_{t_k}}\\
&=& \lim_{k\rightarrow \infty} f(\lambda(\theta_{\underline u}( q_k)) - r \mathbf{1} + V_k )  \\
&=& \lim_{k\rightarrow \infty} f(\lambda(\theta_{\underline u}(q ) )- 2r \mathbf{1} + V_k  +  ( r \mathbf{1} + \lambda( \theta_{\underline u}( q_k )- \lambda(\theta_{\underline u}( q)  )  )) \\   
&\geq &\lim_{l \rightarrow \infty} f(\lambda(\theta_{\underline u}( q))- 2r \mathbf{1} + l (0, ..., 0, 1) ) \\   
&=&\infty
\end{eqnarray*}
by (\ref{fcon1}),  as  $r \mathbf{1} + \lambda( \theta_{\underline u}) ( q_k )- \lambda(\theta_{\underline u}( q)  ) \in \Gamma_n$ for sufficiently large $k>>1$. 
Contradiction.

\medskip

\item Suppose $f_{\infty}$ is bounded.
By continuity, there exists sufficiently small $r>0$, such that for any $p\in M$ and $t\in \cT$, 
$$\lambda(\theta_{\underline{u}}(p)) - 2r \mathbf{1} \in \Gamma, $$
$$\min_{p\in M}\left( e^{-\psi_t(p) - c_t} f_{\infty}(\lambda(\theta_{\underline{u}}(p))- 2r \mathbf{1}) \right) > 1+ \frac{\delta}{2}.$$
   Then similar to case (1), we have 
\begin{eqnarray*}
&&e^{\psi_{t'}(q) +  c'} \\
&=&\lim_{k \rightarrow \infty} e^{\psi_{t_k}(q_k)+ c_{t_k}} \\
&=& \lim_{k\rightarrow \infty} f(\lambda(\theta_{\underline u}( q_k)) - r \mathbf{1} + V_k )  \\
&\geq &\lim_{l \rightarrow \infty} f(\lambda(\theta_{\underline u}( q )- 2r \mathbf{1} + l(0,...,0,1) ))\\   
&= &f_{\infty} (\lambda(\theta_{\underline u}(q)- 2r \mathbf{1}))\\
&\geq & (1+ \delta/2) e^{\psi_{t'}(q) + c'}.
\end{eqnarray*}
 Contradiction. 
\end{enumerate}
 The proposition immediately follows.
\end{proof}

 We now prove uniqueness for the solutions of equation (\ref{maineqn}). 

\begin{lemma}\label{uniq1} Suppose  $u  \in \cE$ solves equation (\ref{maineqn}). Then
$$F[u]= \sigma e^\psi,$$
where $\sigma$ is the sup-slope for (\ref{maineqn}). 

\end{lemma} 

\begin{proof}  By definition of $u$ and $\sigma$,  
$$\sigma \leq \max_M e^{-\psi} F[u] =e^{-\psi} F[u].$$ Suppose $\sigma < e^{-\psi} F[u_1].$ Then there exists $u'\in \cE$ such that 
$$ \sigma \leq \max_M e^{-\psi} F[u'] < e^{-\psi} F[u].$$
Applying the maximum principle at the minimum point of $\phi=u'-u$, we have
$$ e^{-\psi} F[u] \leq e^{-\psi} F[u'] \leq \max_M e^{-\psi} F[u'] < e^{-\psi} F[u].$$
Contradiction. Therefore $e^{-\psi} F[u]=\sigma$. 
\end{proof}

\begin{corollary}
Suppose   equation (\ref{maineqn}) admits a solution in $\cE$. Then
$$\sup_{u\in \cE} \min_M e^{-\psi} F[u] = \inf_{u\in \cE} \max_M e^{-\psi} F[u].$$

\end{corollary}

\begin{proof} Let $v\in\cE$ be the solution to (\ref{maineqn}). By Lemma \ref{uniq1},  
$$e^{-\psi} F[v]=\inf_{u\in \cE} \max_M e^{-\psi} F[u] \leq \sup_{u\in \cE} \min_M e^{-\psi} F[u].$$
Suppose $e^{-\psi} F[v] < \sup_{u\in \cE} \min_M e^{-\psi} F[u] $. Then  there exists $w\in \cE$ such that 
$$ e^{-\psi} F[v]<  \min_M e^{-\psi} F[w].$$ Applying the maximum principle to the maximal point of $\phi= w- v$, we have
$$e^{-\psi} F[v]<   \min_M e^{-\psi} F[w] \leq e^{-\psi} F[w] \leq e^{-\psi} F[v].$$
Contradiction.
\end{proof}

\begin{lemma} \label{uniqq} Suppose $u_1, u_2\in \cE$ are solutions of equation (\ref{maineqn}). Then 
$$u_2- u_1 = constant. $$

\end{lemma}

\begin{proof} Let $v_t = (1-t) u_1 + t u_2$. Then in local coordinates, we have
$$0=F[u_2]- F[u_1]=\int_0^1 \frac{d}{dt} F[v_t] dt=\left(\int_0^1 F_{ij}[v_t] dt \right)  (u_2 - u_1)_{i j}. $$
 By the concavity and monotonicity of $f$, $L = \left(\int_0^1 F_{ij}[v_t] dt \right)\frac{\partial^2}{\partial x_i \partial x_j} $ is a second order elliptic differential operator on $M$. $L (u_2 - u_1)=0$ implies that $u_2-u_1$ must be a constant by the strong maximum principle. 
\end{proof}


\section{Proof of Theorem \ref{mainthm1} and Theorem \ref{mainthm2}}

In this section, we will prove the main theorems of the paper. It suffices to prove $\cT$ defined in (\ref{solvT}) is closed, for the continuity method (\ref{contim}).

\begin{lemma} \label{consolvv} For any $k>0$, there exists $C_k>0$ such that for any $t\in \cT$ in (\ref{solvT}), 
the solution $\phi_t$ to equation (\ref{contim}) satisfies
 \begin{equation}
 \|\phi_t\|_{C^{k}(X)}\leq C_k.
 \end{equation}
\end{lemma}

\begin{proof} By Proposition \ref{pinsubsol}, $\underline u$ is a $\cC_{e^{\psi_t+c_t}, r, R}$-subsolution for some fixed $0<r<R$ with $\psi_t$ and $c_t$ uniformly bounded for all $t\in \cT$. The lemma immediately follows from the a priori estimates in \cite{Sz}.
\end{proof}

The closedness of $\cT$ immediately follows and we are finally able to solve equation (\ref{maineqn}). 

\begin{corollary}  \label{consolv} There exists a unique $u \in \cE$ solving the equation 
$$F[u]=  e^{\psi+c}.$$

\end{corollary}

We have now proved that $(4) \Rightarrow (1)$ in Theorem \ref{mainthm1}.  

\begin{lemma} \label{3to2} $(1) \Rightarrow (3).$

\end{lemma}

\begin{proof} Let $u$ be the solution of equation (\ref{maineqn}).  Since $u$ is smooth,  
$$e^{-\psi} f_{\infty}(\lambda(\theta_u )) > e^{-\psi} f(\lambda(\theta_u))=e^c$$
on $M$ for a fixed $c\in \mathbb{R}$ by (\ref{fcon1}). The lemma automatically holds if $\min_{M} e^{-\psi} f_{\infty}(\lambda(\theta_u ) )=\infty$. Otherwise, suppose there exist a sequence of points $q_k$ such that 
\begin{equation}\label{1to3}
\lim_{k\rightarrow \infty}  e^{-\psi(q_k)} f_{\infty}(\lambda(\theta_u ))  (q_k) = e^c. 
\end{equation}
without loss of generality, we can assume $q_k \rightarrow q$. There exist $\epsilon>0$ and  $K>>1$ such that  for each $1\leq i \leq n$,  
$$e^{-\psi(q)} f(\lambda_1(q), ..., \lambda_{i-1}(q), K, \lambda_{i+1}(q), ... \lambda_n(q) )> e^c + 2\epsilon,$$
where $\lambda_1, ..., \lambda_n$ are the eigenvalues of $\theta_u$. Then for all sufficiently large $k$, 
$$e^{-\psi(q_k)} f(\lambda_1(q_k), ..., \lambda_{i-1}(q_k), K, \lambda_{i+1}(q_k), ... \lambda_n(q_k)) > e^c + \epsilon.$$
by continuity of $f$ and $\psi$. We immediately have %
$$e^{-\psi(q_k)} f_{\infty}(\lambda(\theta_u ))  (q_k) > e^c+\epsilon $$
by (\ref{fcon1}) for all sufficiently large $k$.
This contradicts (\ref{1to3}). 
\end{proof}

 The choice $u$ in (3) is automatically a sub-solution associated with the sup-slope and so $(3) \Rightarrow (2)$. If we choose $\underline u$ in (2) to be the sub-solution associated to the sup-slope in (3), there must exist a super solution $\overline u$ so that (4) holds.  Hence we have completed the proof of Theorem \ref{mainthm1} by showing 
$$(4) \Rightarrow (1) \Rightarrow (3) \Rightarrow (2) \Rightarrow (4).$$ 
Combining the uniqueness result of Lemma \ref{uniq1},  we have now completed the proof of Theorem \ref{mainthm1}. Theorem \ref{mainthm2} is  proved by the same argument.



\section{The non-constant $J$-equation on singular varieties}

We will prove Theorem \ref{mainthm4} in this section. Let $X$ be an $n$-dimensional projective normal variety with $\psi\in C^\infty(X)$ and two K\"ahler classes $\alpha$ and $\beta$. Under the assumptions of Theorem \ref{mainthm4}, there exist a pair of   super-solution and subsolution  $ \chi, \underline \chi \in \cK(\alpha)$ for the singular $J$-equation  (\ref{sinjeqn})
$$
\theta^{n-1} \wedge \omega = e^{\psi + c} \theta^n 
$$
on $X$, i.e., there exists $\overline\epsilon>0$ such that 
\begin{equation}\label{subjs}
(n-1)e^{-\psi} \underline\chi^{n-2} \wedge \omega < n(\overline \xi  - 2\overline\epsilon) \underline \chi^{n-1}, 
\end{equation}
where
\begin{equation}\label{subjs0}
\overline \xi = \inf_{X^\circ} \left( e^{-\psi} \frac{ \chi^{n-1}\wedge \omega}{  \chi^n}\right). 
\end{equation}
In particular, the subsolution condition (\ref{subjs}) also implies that   for any $1\leq m \leq n-1$,
\begin{equation}\label{subjs2}
m e^{-\psi} \underline\chi^{m-1} \wedge \omega < n(\overline \xi  - 2\overline \epsilon) \underline \chi^{m}. 
\end{equation}

We will fix a very ample line bundle $L$  over $X$ . Then the linear system of $L$ induces a projective embedding
$$\Phi: X \rightarrow \mathbb{CP}^N$$
with $\Phi^* \cO_{\mathbb{CP}^N}(1)= L. $  Let $\cS$ be the singular set of $X$ and let
\begin{equation}\label{resol}
\pi: Y \rightarrow X
\end{equation} 
be  a log resolution of singularity
such that the exceptional locus $\Xi$ of $\pi$ coincides with $\pi^{-1}(\cS)$.

The singular set $\cS$ can be decomposed into a stratification of finitely many smooth components. More precisely, we write $\cS$ as a disjoint union
\begin{equation}
\cS = \sqcup_{i=0}^{n-2} \left( \sqcup_{j=1}^{J(i)} \cS_{i, j} \right),~\dim_{\mathbb{C}} \cS_{i, j} = i
\end{equation}
such that each strata $\cS_{i, j}$ is connected and smooth. In particular, $\cS_{i, j}$ is either projective or quasi-projective  in $\mathbb{CP}^N$.

 For each $\cS_{i, j}$, let $\cJ_{i, j}$ be the ideal sheaf of $\overline{\cS_{i, j}} \subset \mathbb{CP}^N$.  Then for some sufficiently large $m\in \mathbb{Z}^+$, $L^{\otimes m} \otimes \cJ_{i, j}$ is globally generated. Therefore there exist $\sigma_1, ..., \sigma_{N_m}$ such that all of them vanish along $\cS_{i, j}$ and at any point $p \in \cS_{i, j}$, $\cS_{i, j}$ locally can be defined in $\mathbb{CP}^N$ by $N+1-i$ sections of $\{ \sigma_1, ..., \sigma_{N_m}\}$. In particular, we can choose local holomorphic coordinates $z=(z_1, ..., z_{N+1})$ near $p$ in $\mathbb{CP}^N$ such that $p=0$ and these $N+1-i$ sections  are given by $\{z_{i+1}=0\}$, $\{z_{i+2}=0\}$, ..., $\{z_{N+1}=0\}$. We choose $h$ to be a smooth hermitian metric on $L^{\otimes m}$. We define
\begin{equation}
\rho_{i,j} = \sum_{k=1}^{N_m} |\sigma_k|^2_h .
\end{equation}

\begin{lemma} \label{interior} For any $\cS_{i, j}$, we have 
\begin{equation}
\rho_{i, j} |_{ \overline{\cS_{i,j}}}  =0, ~ \rho_{i, j} |_{ \mathbb{CP}^N \setminus \overline{\cS_{i,j}}}  >0, ~ \left( \partial \rho_{i, j} \right) |_{ \overline \cS_{i,j}}  =0. 
\end{equation}
Furthermore,  for any $p\in \cS_{i, j}$ and local smooth vector fields $V^T$, $V^\perp$ near $p$ in $\mathbb{CP}^N$ with $0\neq V^T|_p \in T_p\cS_{i, j}$ and $0\neq V^\perp |_p \in  N_p\cS_{i, j}$, we have at $p$ 
\begin{equation}
\ddbar \rho_{i, j} (V^T, \overline{ V^T}) =0, ~\ddbar \rho_{i, j} (V^\perp, \overline{ V^\perp} ) >0.
\end{equation}

\end{lemma}

We  define
$$\phi_{i, j}   =  \epsilon  (\rho_{i,j})^{1-\delta}, ~\gamma_{i, j}   = \chi +\ddbar \phi_{i, j}$$
for   $\epsilon' >0$ and $0<\delta<1$ to be determined. 
\begin{lemma} For any $\delta\in (0, 1)$, there exists $\epsilon_0>0$ such that for any $(i, j)$ and $0<\epsilon< \epsilon_0$,
$\gamma_{i, j}$ is a K\"ahler current on $X$, i.e., it is bounded below by a proportion of $\omega$. In particular, $\gamma_{i, j}$ is smooth on $X^\circ$.
\end{lemma}
 In particular,  $\gamma_{i, j}$ can locally be viewed as the restriction of a singular K\"ahler metric with cone singularities along $\cS_{i, j}$ in $\mathbb{CP}^N$.  %
 We can now perturb $\gamma_{i, j}$ and $\phi_{i, j}$ by 
$$\phi_{i, j, \varepsilon}=  \epsilon  (\rho_{i,j} + \varepsilon)^{1-\delta}, ~ \gamma_{i, j, \varepsilon} = \chi +\ddbar \phi_{i, j, \varepsilon} .$$

The exceptional locus $\Xi_{i, j}=\pi^{-1}(\cS_{i,j})$ can also be decomposed into a finite stratification
$$\Xi_{i, j} = \sqcup_{k=i}^{n-1} \left( \sqcup_{l=1}^{L(i,j,k)} \Xi_{i, j, k, l} \right),~\dim_\mathbb{C}\Xi_{i, j, k,l}=k $$
such that each $\Xi_{i, j, k, l}$ is smooth and the restricted morphism of (\ref{resol})
$$\pi|_{\Xi_{i, j, k, l}}: \Xi_{i, j, k, l} \rightarrow \cS_{i, j}\subset X \subset \mathbb{CP}^N$$
is surjective and has constant rank $k-i$.

For any point $p\in \Xi_{i, j, k, l} \subset Y$, one can find a local chart $U \subset Y$ of $p$ such that %
$$p=0\in \mathbb{C}^n, ~\Xi_{i, j, k, l} =\mathbb{C}^k\times \{0\} \cap U \subset \mathbb{C}^n. $$
 The map $\pi|_{\Xi_{i, j, k, l}}$ in $U$ is the restriction of the projection from $\mathbb{C}^k=\mathbb{C}^i \times \mathbb{C}^{k-i}$ to $\mathbb{C}^i$ $(k\geq i)$. We then let 
 $$\cF_{i, j, k, l, p}= \mathbb{C}^i\times \{0\} \cap U.$$ In particular, $\pi|_{\cF_{i, j, k, l, p}}:  \rightarrow \cS_{i, j}$ is non-degenerate and so locally biholomorphic.

For convenience, we identify $\omega$, $\chi$, $\underline{\chi}$, $\gamma_{i, j}$ with $\pi^*\omega$, $\pi^*\chi$, $\pi^*\underline{\chi}$, $\pi^* \gamma_{i, j}$. 

\begin{lemma} \label{tranest1} For any point $p \in \Xi_{i, j, k, l}$ and locally smooth vector fields $V^T$, $V^\perp$ on $Y$ near $p$ with $0\neq V^T|_p \in T_p \cF_{i, j, k, l, p}$ and $0\neq V^\perp |_p \in N_p \Xi_{i, j, k, l, p}$,  we have 

\begin{equation}\label{tang1}
\gamma_{i, j} (V^T, \overline{V^T}) >0
\end{equation}
near $p$ and 
\begin{equation}\label{perp1}
\lim_{q\rightarrow p } \frac{ \gamma_{i, j} (V^\perp, \overline{V^\perp})(q) }{ \omega  (V^\perp, \overline{V^\perp}) (q)}= \infty.
\end{equation}

\end{lemma}

\begin{proof} We omit the indices $i, j, k, l$ for convenience. Let $\cF_p^\perp =\left( \{0\}\times \mathbb{C}^{n-i} \right)\cap U$ that is orthogonal to $\cF_p= \mathbb{C}^i \times \{0\}$ with $p=0\in \mathbb{C}^n$. Since $\gamma_{i, j}$ is a K\"ahler current on $X$, (\ref{tang1}) immediately follows. 

 Let $\hat V$ be a local smooth vector field near $\pi(p)$  defined by $\hat V = (\Phi\circ\pi)_* V^\perp \in T_{\pi(q)} X\subset T_{\Phi(\pi(q))} \mathbb{CP}^{N}$ for  $q\in \cF_p^\perp$.   $\hat V|_{\pi(p)} \neq 0$  and it does not lie in $T_{\pi(p)} \cS_{i, j}$.  Since $\gamma_{i, j}$ has cone singularities in each transversal direction of $\cS_{i,j}$ and $\omega$ is smooth, we have 
$$\lim_{q \rightarrow p} \frac{\gamma_{i, j}( \hat V, \overline{  \hat V})(\pi(q))} {\omega( \hat V^\perp, \overline{\hat V^\perp})(\pi(q))} = \infty$$
and (\ref{perp1}) follows.
\end{proof}

We now choose a singular K\"ahler metric $\theta_{con}=\theta(\delta)$ on $Y$ with conical singularities of normal crossings along the exceptional locus of $\pi$ with angle $2(1-\delta)\pi$. More precisely, let $E=\sum_{i=1}^I E_i$ be an effective $\mathbb{Q}$-divisor whose support coincides with the exceptional locus $\Xi$ of $\pi$. Let $\sigma_{E_i}$ be the defining section of $E_i$. We can choose  a smooth hermitian metric $h_{E_i}$ (suitably scaled) of the line bundle associated to $E_i$ such that   
\begin{equation}\label{helo}
 \theta_{con} = \theta_Y + \sum_{i=1}^I \ddbar |\sigma_{E_i}|^{2(1-\delta)}_{h_{E_i}}, ~ \theta_{con,\varepsilon} = \theta_Y + \sum_{i=1}^I \ddbar (|\sigma_{E_i}|^2_{h_{E_i}}+\varepsilon)^{(1-\delta)}
 \end{equation}
for sufficiently small $\varepsilon>0$.

\begin{lemma}\label{tranest2} For any point $p\in \Xi_{i, j, k, l}$ any $K>1$ and any $t>0$, there exists a neighborhood $U \subset Y$ of $p$ such that for any $0 \leq r \leq n$, 

\begin{equation}\label{ktform}
 t\chi^{r-1} \wedge \gamma_{i, j} \wedge \theta_Y ^{n-r}  + \chi^r \wedge \theta_Y^{n-r-1}\wedge  \theta_{con}  \geq K \chi^r \wedge \theta_Y^{n-r}
 \end{equation}

\end{lemma}

\begin{proof} Suppose the exceptional locus near $p$ is given by the intersection of $E_{k+1}$, $E_{k+2}$, .., $E_{n}$ with $m \leq n-1$. We choose local holomorphic vector fields $V_1, ...., V_n$ near $p$ such that   $\langle V_1|_p, ..., V_n|_p \rangle =T_pY$   and   $\langle V_1|_p, ... V_m|_p \rangle =  T_p \Xi_{i, j, k,l}$. We can also assume at $p$ they are orthonormal with respect to $\theta_Y$, and so by choosing a sufficiently small neighborhood $U$ of $p$, we can assume that $V_{k+1}, ...., V_n$ are transversal to $\Xi_{i, j, k, l}\cap U$.  

Apply $V_1\wedge \overline{V_1} \wedge .... \wedge V_n \wedge \overline{V_n}$ to both sides of (\ref{ktform}).  Suppose one of $V_{k+1}$, ..., $V_{n}$ contracts with  $ \theta_Y^{n-r}$ on the right, then the inequality immediately holds by contracting the same tangent vector with $\theta_{con} $. 

Suppose a subset of $V_{1}$, ..., $V_{k}$ contracts with $\theta_Y^{n-r} $ in $\chi^r\wedge\theta_Y  ^{n-r} $, (which implies that $k\geq n-r$). Then for any subspace $\mathbb{C}^{r}\cap U$ passing through $p$ and containing a transversal directions of $\Xi_{i, j, k,l}$,  
$$\frac{\chi^{r-1}\wedge \gamma_{i, j}}{\chi^{r}} \rightarrow \infty$$
near $p$ by Lemma \ref{tranest1}.   The lemma follows immediately by choosing $U$ sufficiently small after fixing $t$ and $K$.
 \end{proof}

We define
$$\gamma=\gamma(\delta) =\left( \sum_{i, j} 1 \right)^{-1}  \sum_{i, j} \gamma_{i, j}, ~\gamma_\varepsilon =\left( \sum_{i, j} 1 \right)^{-1}  \sum_{i, j} \gamma_{i, j, \varepsilon} . $$

Then following corollary immediately follows from Lemma \ref{tranest2}..
\begin{corollary}\label{tranest3} For any $K>1$ and $t>0$, there exists a neighborhood $U$ of $\cE$ such that
$$ t\chi^{r-1} \wedge \gamma  \wedge \theta_Y ^{n-r}  + \chi^r \wedge \theta_Y^{n-r-1}\wedge  \theta_{con}    \geq K \chi^r \wedge \theta_Y^{n-r}$$
in $U$ in the sense of currents.

\end{corollary}

The following corollary follows from Lemma \ref{tranest2} for the perturbed forms $\gamma_{con, \varepsilon}$ and $\theta_{con, \varepsilon}$.

\begin{corollary}\label{tranest4} For any $K>1$ and any $t>0$, there exist a neighborhood $U$ of $\cE$  and $\varepsilon_0>0$ such that for $0<\varepsilon < \varepsilon_0$,  we have 
$$ t\chi^{r-1} \wedge \gamma_\varepsilon \wedge \theta_Y ^{n-r}  + \chi^r \wedge \theta_{con} ^{n-r-1}\wedge  \theta_{con, \varepsilon}  \geq K \chi^r \wedge \theta_Y ^{n-r}$$
in $U$.

\end{corollary}

\begin{lemma} There exist $\delta_0>0$, $p>1$ and $C>0$ such that for all $\delta< \delta_0$, 

$$\left\| \frac{(\gamma +   \theta)^n}{\theta_Y^n} \right\|_{L^p(Y, \theta_Y^n)} \leq C. $$

\end{lemma}

\begin{proof} For any sufficiently small $\epsilon' \in (0, 1)$, there exist $\delta_0>0$ and $C>0$ such that for all $\delta<\delta_0$, 
$$ \gamma \leq C\left( \sum_{i=1}^I  |\sigma_{E_i}|^2_{h_{E_i}} \right)^{-\epsilon'}\theta_Y, ~\theta_{con}  \leq C \left(\sum_{i=1}^I  |\sigma_{E_i}|^2_{h_{E_i}} \right)^{-\epsilon'}\theta_Y, $$
where $  \sigma_{E_i}$ and $  h_{E_i}$ are defined in (\ref{helo}). Then  
$$\int_Y \left| \frac{(\gamma +   \theta)^n}{\theta_Y^n} \right|^p \theta_Y^n \leq C \int_Y |\sigma_E|^{-2\epsilon' p} \theta_Y^n.$$
and the lemma follows immediately by choosing a suitable $p>1$. 
\end{proof} 

The following corollary follows from Lemma \ref{helo} by for the perturbed forms  for $\gamma_\varepsilon$ and $\theta_{con, \varepsilon}$. 

\begin{corollary}\label{lpp0}

There exists $\delta_0>0$, $p>1$ and $C>0$ such that for all $0<\delta< \delta_0$ and $0<\varepsilon<1$, 

$$\left\| \frac{(\gamma_\varepsilon +  \theta_{con, \varepsilon} )^n}{\theta_Y^n} \right\|_{L^p(Y, \theta_Y^n)} \leq C. $$

\end{corollary}

We will fix a sufficiently large constant $A>>1$.

\begin{proposition} \label{appsupsol} There exists $C>0$ such that for any $t\in (0, 1)$, there exist $s_0 >0$ and $\varepsilon_0>0$ such that for all $s\in (0, s_0)$ and $\varepsilon\in (0, \varepsilon_0)$, we have   %
\begin{equation}\label{supcomp}  
(\chi+ At \chi + As\theta_Y)^{n-1} \wedge (\omega + t\gamma_\varepsilon + s \theta_{con, \varepsilon})  >( \overline \xi - \overline\epsilon) e^{\psi}  (\chi+At\chi + As \theta_Y)^n . 
\end{equation}
and 
\begin{equation}\label{lpp}
\left\| \frac{(\omega+ t \gamma_\varepsilon + s \theta_{con, \varepsilon})^n}{\theta_Y^n} \right\|_{L^p(Y, \theta_Y^n)} \leq C, %
\end{equation}
where $\overline \xi$ and $\overline \epsilon$ are given by (\ref{subjs0}) and (\ref{subjs}).

\end{proposition}

\begin{proof} We will compare terms containing $s^k$ in both  
$$(I)=(\chi+ At \chi + As\theta_Y)^{n-1} \wedge (\omega + t\gamma_\varepsilon + s \theta_{con, \varepsilon}) $$ and $$(II) = (\chi+At\chi + As \theta_Y)^n$$  for $0\leq k \leq n$. 

\begin{enumerate}

\smallskip

\item $k=0$: We have 
$(1+At)^{n-1}\chi^{n-1} \wedge(\omega+t\gamma_\varepsilon)$ in $(I)$ and 
$(1+At)^n\chi^n$ in $(II)$. Their ratio is given by 
$$ (1+At)^{-1} \left(  \frac{ \chi^{n-1} \wedge(\omega+t\gamma_\varepsilon)} { \chi^n} \right)>(1+At)^{-1} \left(  \frac{ \chi^{n-1} \wedge\omega} { \chi^n} \right) >  (\overline\xi -\overline\epsilon) e^{ \psi}$$
for a fixed $A>>1$ by our choice of $\chi$.

\medskip

\item $k=n$: We have  $ A^{n-1} \theta_Y^{n-1} \wedge \theta_{con, \varepsilon} $ in $(I)$ and 
$A^n \theta_Y ^n$ in $(II)$. We can find a sufficiently small neighborhood $U_1$ of $\cE$ such that the ratio of  $ A^{n-1} \theta_Y^{n-1} \wedge \theta_{con, \varepsilon} $ and $A^n \theta_Y ^n$ is given by 
$$A^{-1} \frac{ \theta_Y^{n-1} \wedge \theta_{con, \varepsilon}}{ \theta_Y^n} \geq \overline\xi e^{ \psi}.$$
for all sufficiently small $\varepsilon>0$, by the choice of $\theta_{con, \varepsilon}$.

\medskip

\item $0< k < n$: We have  $$\left( C^{n-1}_k (1+At)^{n-k-1}  A^k   \chi^{n-k-1}\wedge(\omega+ t\gamma_\varepsilon)\wedge    \theta_Y^k    +  C^{n-1}_{n-k} (1+At)^{n-k} A^{k-1}      \chi^{n-k} \wedge \theta_Y^{k-1} \wedge  \theta_{con, \varepsilon} \right)$$ in $(I$) and 
$C^n_k (1+At)^{n-k} A^k         \chi^{n-k} \wedge \theta_Y^k$ in $(II)$. We can find a sufficiently small neighborhood $U_2=U_2(A, t)$ of $\cE$  so that  their ratio is given by

\begin{eqnarray*}
&&\frac{ C^{n-1}_k (1+At)^{n-k-1}  A^k   \chi^{n-k-1}\wedge(\omega+ t\gamma_\varepsilon)\wedge    \theta_Y^k    +  C^{n-1}_{n-k} (1+At)^{n-k} A^{k-1}      \chi^{n-k} \wedge \theta_Y^{k-1} \wedge \theta_{con, \varepsilon}}{ C^n_k (1+At)^{n-k} A^k         \chi^{n-k} \wedge \theta_Y^k}       \\
&\geq& c(A, t) \frac{ t \chi^{n-k-1}\wedge \gamma_\varepsilon \wedge \theta_Y^k +   \chi^{n-k} \wedge \theta_Y^{k-1}\wedge  \theta_{con, \varepsilon}}{\chi^{n-k}\wedge \theta_Y^k}\\
&\geq& \overline \xi e^\psi
\end{eqnarray*}
for all sufficiently small $\varepsilon>0$, by Corollary \ref{tranest4}.

\end{enumerate}

Then for any sufficiently small $s>0$, we can assume that on $Y\setminus \left( U_1\cap U_2 \right)$, the terms both on the top and the bottoem in (2) and (3) will be much smaller than $\chi^n$ since they contain $s^k$. Estimate (\ref{supcomp})  then easily follows. Estimate (\ref{lpp}) directly follows from Corollary \ref{lpp0} since $\omega$ is bounded by a fixed multiple of $\theta_Y$. We have now completed the proof of the proposition.
\end{proof}

\begin{proposition} \label{appsubsol} There exist $s_0 >0$ and $\varepsilon_0>0$ such that for any $s\in (0, s_0)$, $\varepsilon\in (0, \varepsilon_0)$, we have

\begin{equation}\label{subcomp}
(\underline\chi+ At \gamma_\varepsilon + As \theta_{con, \varepsilon})^{n-2} \wedge (\omega + t\gamma_\varepsilon + s  \theta_{con, \varepsilon}) < \frac{n}{n-1} (\overline \xi - 2\overline\epsilon) e^{\psi}  (\underline\chi+At\gamma_\varepsilon + As \theta_{con, \varepsilon})^{n-1}. 
  \end{equation}

\end{proposition}

\begin{proof}

We will compute the ratio of terms  in 
$$(I)=(\underline\chi+ At \gamma_\varepsilon + As \theta_{con, \varepsilon})^{n-2} \wedge (\omega + t\gamma_\varepsilon + s  \theta_{con, \varepsilon})$$ and $$(II)=(\underline\chi+At\gamma_\varepsilon + As \theta_{con, \varepsilon})^{n-1}$$  that contains $s^a t^b $ for given $a, b\geq 0$. 

\begin{enumerate}

\item $a=b=0$: For any $p\in Y\setminus \cE$ and  $(n-1)$-dimensional subspace of $H \subset T_p Y$,  we have  

$$ \left. \frac{ \underline\chi^{n-2} \wedge\omega}{\underline\chi^{n-1}}\right|_H  <   \frac{n}{n-1}(\overline \xi - 2\overline\epsilon) e^{ \psi}$$
by (\ref{subjs}). 

\medskip

\item $1\leq a+b\leq n-2$:  We first consider the terms containing $\underline \chi^{n-2-a-b }\wedge \omega$ in (I) and those containing $\underline\chi^{n-1-a-b}$ in (II). For any $p\in Y\setminus \cE$ and  $(n-1-a-b)$-dimensional complex subspace of $H \subset T_p Y$, their ratio is given by
\begin{eqnarray*}  
&& \left. \frac{ C^{n-2}_{a+b}  ~ \underline \chi^{n-2-a-b}\wedge \omega\wedge \gamma_\varepsilon^a\wedge  \theta_{con, \varepsilon}^b   }{ C^{n-1}_{a+b}~ \underline \chi^{n-1-a-b}\wedge \gamma_\varepsilon^a\wedge   \theta_{con, \varepsilon}^b    } \right|_H  \\
&=& \left.  \frac{ (n-1-a-b) \underline \chi^{n-2-a-b}\wedge \omega\wedge  \gamma_\varepsilon^a\wedge  \theta_{con, \varepsilon}^b   }{ (n-1) \underline \chi^{n-1-a-b}\wedge \gamma_\varepsilon^a\wedge   \theta_{con, \varepsilon}^b    }  \right|_H \\
&<&\frac{n}{n-1}  (\overline \xi - 2\overline\epsilon)  e^{ \psi}
\end{eqnarray*}
by (\ref{subjs2}).

We then consider the terms  containing no $ \omega$ in (I) and those containing $\underline\chi^{n-1-a-b}$ in (II). Their ratio is given by

$$ \left.  \frac{  A^{a+b-1}  \underline \chi^{n-1-a-b}\wedge \gamma_\varepsilon^a\wedge  \theta_{con, \varepsilon}^b   }{ A^{a+b}  \underline \chi^{n-1-a-b}\wedge \gamma_\varepsilon^a\wedge   \theta_{con, \varepsilon}^b    } \right|_H  =A^{-1}$$
for any $p\in Y\setminus \cE$ and  $(n-1-a-b)$-dimensional subspace of $H \subset T_p Y$

By choosing $A$ sufficiently large, we can conclude that the ratio of terms containing $t^a s^b$ on the numerator and the denominator will be less than $\frac{n}{n-1}  (\overline \xi - 2\overline\epsilon)  e^{ \psi}$.

\medskip

\item $a+b=n-1$: The ratio of terms containing $t^a s^b$  in (I) and (II) will be exactly $A^{-1}$ and by choosing sufficiently large $A>1$, we can conclude that it is less than $\frac{n}{n-1} (\overline \xi - 2\epsilon)  e^{ \psi}$.
\end{enumerate}
We have now completed the proof of the proposition.
\end{proof}

We now let $\omega_t = \omega + t\gamma_\varepsilon + s  \theta_{con, \varepsilon}$ and $ \chi_t =\underline  \chi +A t\gamma_\varepsilon + As  \theta_{con, \varepsilon}$ by choosing $0< \varepsilon,  s<<t$. It is obvious that $\omega_t$ and $\chi_t$ are bounded by each other uniformly. 

\begin{lemma} \label{mequ} There exist $p>1$ and $C>0$ such that for all $t\in (0, 1)$, 
\begin{equation}\label{mequ1}
 C^{-1} \omega_t \leq \chi_t \leq C \omega_t, ~ \left\| \frac{\omega_t^n}{\theta_Y^n} \right\|_{L^p(Y, \theta_Y^n)} \leq C. 
 \end{equation}

\end{lemma}

Then we consider the following $J$-equation
\begin{equation}\label{perjeqn}
\frac{(\chi_t + \ddbar u_t)^{n-1} \wedge \omega_t}{(\chi_t + \ddbar u_t)^n} = e^{\psi+ c_t}, ~\max_X \varphi_t =0.
\end{equation}
on $Y$ for $t\in (0, 1)$. Equation (\ref{perjeqn}) is equivalent to the following equation 
\begin{equation}\label{jnormal}
F[u_t]= f(\lambda(\chi_t + \ddbar u_t))=e^{-\psi-c_t}, 
\end{equation}
where $f(\lambda) =\frac{n S_n}{S_{n-1}} $ and $\lambda_{\chi_t+\ddbar u_t}$ is the set of eigenvalues of $\chi_t+\ddbar u_t$ with respect to $\omega_t$. 
Obviously $f$ satisfies the assumption of Theorem \ref{mainthm2} for the positive cone $\Gamma_n$. 
Recall that the sup-slope $\sigma_t$ for equation (\ref{jnormal}) is given by
\begin{equation}\label{jsslo}
\sigma_t = \inf_{\phi\in C^\infty(Y)\cap {\rm PSH}(X, \chi_t)} \max_Y \left( e^{\psi} \frac{(\chi_t + \ddbar \phi)^n}{(\chi_t + \ddbar \phi)^{n-1} \wedge \omega_t} \right), %
\end{equation}
In particular, the $J$-slope for equation (\ref{perjeqn}) is given by 
$$\xi_t = \sigma_t^{-1}= e^{c_t}. $$

\begin{lemma}
For any $t\in (0,1)$,   $\underline u=0$ is both a super-solution and a sub-solution for  equation (\ref{jnormal}) satisfying
\begin{equation}\label{jsubsup}
\max_Y e^{\psi} f(\lambda(\chi_t)) < \min_{Y}  e^{\psi} f_{\infty}(\lambda(\chi_t)) -\overline \epsilon. 
\end{equation}
 In particular, for any $0<t <1$, equation (\ref{perjeqn}) admits a unique solution $u_t\in C^\infty(Y)\cap {\rm PSH}(Y, \chi_t)$. 

\end{lemma}

\begin{proof} The estimate (\ref{jsubsup})  follows from   Proposition \ref{appsupsol} and Proposition \ref{appsubsol} by the construction of $\omega_t$ and $\chi_t$. The lemma is then proved by applying Theorem \ref{mainthm2}. 
\end{proof}

\begin{lemma} There exists $C>0$ such that for any $0<t<1$, the sup-slope $\sigma_t $ for equation (\ref{perjeqn}) satisfies
$$  C^{-1}\leq \sigma_t \leq C.$$
\end{lemma}

\begin{proof}  It suffices to show that $\sigma_t$ is uniformly bounded below, or equivalently, $c_t$ is uniformly bounded above. Since equation (\ref{perjeqn}) can be solved for each $t\in (0,1)$, we 
 let $p_t$ be the minimal point of $u_t$. By the maximum principle,   we have 
$$e^{c_t} \leq \left. e^{-\psi} \frac{\chi_t^{n-1}\wedge \omega_t}{\chi_t^n}  \right|_p \leq \max_Y \left( e^{-\psi} \frac{\chi_t^{n-1}\wedge \omega_t}{\chi_t^n}  \right) \leq C $$ 
since $\omega_t$ and $\chi_t$ are uniformly equivalent to each other by Lemma \ref{mequ}. 
\end{proof}

The same argument in the proof of Proposition \ref{pinsubsol} gives the following lemma. 

\begin{lemma} There exist $r>0$, $R>0$ such that  $\underline{u}=0$ is a $\cC_{e^{-\psi+c_t}, r, R}$-subsolution for (\ref{jnormal}) for all $t\in (0,1)$. 

\end{lemma}

The recent important work of \cite{GP} extends the $L^\infty$-estimate of \cite{Sz}  to a family of degenerating background K\"ahler metrics with techniques developed in \cite{GPT}. It turns out to be an essential estimate in our application for the $J$-equation  in the singular setting.   

\begin{lemma} Let $u_t$ be the unique solution of equation (\ref{perjeqn}).  There exists $C>0$ such that for all $0<t<1$, we have 
$$\|u_t\|_{L^\infty(Y)} \leq C. $$

\end{lemma}

\begin{proof}  By the uniform estimate (\ref{mequ1}), we can directly apply Theorem 2.1 of \cite{GP} to obtain the uniform $L^\infty$-estimate for $u_t$ since $\underline u=0$ is a $\cC_{e^{-\psi-c_t}, r, R}$-subsolution for (\ref{jnormal}) for fixed $r$, $R>0$ for all $t\in (0,1)$.
\end{proof}

\begin{corollary} There exists a  solution $u \in L^\infty(X)\cap {\rm PSH}(X, \underline \chi)$ to the $J$-equation 
\begin{equation}\label{jonx}
 (\underline \chi + \ddbar u)^{n-1}\wedge \omega = e^{\psi+c} (\underline \chi + \ddbar u)^n.
 \end{equation} 
for some $c\in \mathbb{R}$. Furthermore, there exists $C>0$ such that 
 \begin{equation}\label{kcond}
 \underline \chi + \ddbar u \geq C^{-1} \omega. 
 \end{equation}

\end{corollary}

\begin{proof} Let $\varphi_t$ be the solution to the equation (\ref{perjeqn}). By passing to a subsequence, we can assume that $u_{t_j} \rightarrow u \in L^\infty(Y)\cap {\rm PSH}(Y, \underline \chi)$, $c_{t_j} \rightarrow c_0$ as $t_j \rightarrow 0$. Since $u_t$ is uniformly bounded in $L^\infty(Y)$, 
\begin{eqnarray*}
&& \lim_{t_j \rightarrow 0} (\underline \chi + \ddbar u_{t_j} )^{n-1}\wedge \omega = (\underline \chi + \ddbar u)^{n-1}\wedge \omega \\
&=& \lim_{t_j \rightarrow 0} e^{\psi+c_{t_j}} (\underline \chi + \ddbar u_{t_j})^n = e^{\psi+c_0} (\underline \chi + \ddbar u)^n 
\end{eqnarray*}
holds on $Y$. Since $\underline \chi$ is a K\"ahler form on $X$, $u$ can be pushed forward to $X$ and $u \in L^\infty(X)\cap {\rm PSH}(X, \underline \chi)$. Therefore, $u$ indeeds solves equation (\ref{jonx}).

Since $\chi_t + \ddbar u_t$ is smooth on $Y$, there exists $C>0$ such that for all $t\in (0, 1)$, 
$$\chi_t + \ddbar u_t \geq C^{-1}  \omega_t$$
by equation (\ref{perjeqn}).  
The K\"ahler condition (\ref{kcond}) then immediately follows.
\end{proof}

We have now completed the proof of Theorem \ref{mainthm4}. 
\bigskip
\bigskip

\end{document}